\documentclass[10pt,twocolumn,twoside]{IEEEtran}
\usepackage{cite}
\usepackage{amsmath,amssymb,amsfonts,amsthm}
\usepackage{algorithmic}
\usepackage{graphicx}
\usepackage{subfig}
\usepackage{textcomp}
\usepackage[english]{babel} 
\usepackage[dvipsnames]{xcolor}
\usepackage{url}

\graphicspath{{Figures/}}
\graphicspath{{FiguresSubmitted/}}

\DeclareMathOperator*{\argmin}{arg\,min}
\DeclareMathOperator*{\poa}{PoA}

\newtheoremstyle{teo}
  {3pt} 
  {3pt} 
  {\itshape} 
  {} 
  {\bfseries} 
  {.} 
  {.5em} 
  {} 

\theoremstyle{teo}
\newtheorem{lem}{Lemma}
\newtheorem{prop}{Proposition}
\newtheorem{corol}{Corollary}
\newtheorem{thm}{Theorem}

\newtheoremstyle{defi}
  {3pt} 
  {3pt} 
  {} 
  {} 
  {\bfseries} 
  {.} 
  {.5em} 
  {} 

\theoremstyle{defi}
\newtheorem{defi}{Definition}
\newtheorem{assum}{Assumption}

\newtheoremstyle{rema}
  {3pt} 
  {3pt} 
  {} 
  {} 
  {\bfseries} 
  {.} 
  {.5em} 
  {} 
\theoremstyle{rema}
\newtheorem{rem}{Remark}
\newtheorem{ex}{Example}

\usepackage{color}

\begin{document}
\title{Selfish routing on transportation networks with supply and demand constraints}
\author{Tommaso Toso, \IEEEmembership{Student Member, IEEE}, Paolo Frasca, \IEEEmembership{Senior Member, IEEE}, and~Alain~Y.~Kibangou,~\IEEEmembership{Member,~IEEE}
\thanks{This work has been supported in part by the French National Research Agency in the framework of the ``Investissements d’avenir'' program ANR-15-IDEX-02 and of LabEx PERSYVAL ANR-11-LABX-0025-01. }
\thanks{Tommaso Toso, Paolo Frasca and Alain Kibangou are with Univ. Grenoble Alpes, CNRS, Inria, Grenoble INP, GIPSA-lab, 38000 Grenoble, France (e-mail: tommaso.toso@grenoble-inp.fr; paolo.frasca@grenoble-inp.fr, alain.kibangou@univ-grenoble-alpes.fr). Alain Kibangou is also with Univ.\ of Johannesburg (Auckland Park Campus), Johannesburg 2006, South Africa.
}}

\maketitle

\begin{abstract}
   Traditional non-atomic selfish routing games present some limitations in properly modeling road traffic. This paper introduces a novel type of non-atomic selfish routing game leveraging concepts from Daganzo's cell transmission model (CTM). Each network link is characterized by a supply and demand mechanism that enforces capacity constraints based on current density, providing a more accurate representation of real-world traffic phenomena. We characterize the Wardrop equilibria and social optima of this game and identify a previously unrecognized inefficiency in selfish routing: partially transferring Wardrop equilibria, where only part of the exogenous flow traverses the network.

    %
\end{abstract}

\begin{IEEEkeywords}
Transportation Networks, Game Theory, Cell Transmission Model.
\end{IEEEkeywords}

\section{Introduction}

Non-atomic selfish routing games model the interactions of self-interested players in a network. Each player, associated with an origin-destination (OD) pair, aims to traverse the network from their origin to their destination. The cost of each link in the network depends on the number of players using that link. With full information about the network state, each player seeks to minimize their traversal cost. However, since players act uncoordinatedly and disregard the impact of their actions on others, the resulting equilibrium configurations are inefficient from a total cost perspective. These equilibria are known as Wardrop equilibria \cite{wardrop52,rough}.
A major concern in the analysis of selfish routing is the inefficiency of Wardrop equilibria. Various metrics have been developed to assess this inefficiency, with the Price of Anarchy (PoA) being one of the most commonly used \cite{rough}.

Non-atomic selfish routing games have numerous applications, notably in transportation networks, when users have access to information about the state of the network. This situation is more and more relevant in today's information-rich world. Nowadays, many motorists rely on real-time navigation systems to optimize their routes, significantly influencing traffic flow~\cite{th:cabannes}. Given the increasing awareness and responsiveness of users to traffic congestion, a game-theoretic approach like non-atomic selfish routing games is suitable for modeling contemporary traffic flow behavior.

Despite their usefulness in understanding traffic flow behavior and the impact of real-time routing systems \cite{th:cabannes,ibp,thai}, traditional non-atomic selfish routing games have limitations when applied to traffic networks. First, network links lack capacity constraints, which are essential for capturing typical congestion phenomena. Additionally, link costs are generally modeled as increasing functions of flow, which is inconsistent with established models of the physics of traffic, in which the relationship between traffic flow and travel time is non-monotonic~\cite{funddiag,ttimes}.

To address these limitations, this paper proposes a novel type of non-atomic selfish routing game, leveraging concepts from Daganzo's cell transmission model (CTM) \cite{daganzo1994,daganzo1995}. By considering both link flow and density, we characterize each network link with a supply and demand mechanism that enforces capacity constraints. This mechanism limits the flow through a link based on its current density, allowing us to identify congested sections. Moreover, travel times on each link are directly dependent on its density, aligning with traffic modeling principles. This combined approach provides a more accurate representation of real-world traffic network phenomena. Thanks to this new model, we are able to identify a potential drawback of selfish routing that extends beyond classical PoA analysis and was not recognized in the literature: \emph{partially transferring Wardrop equilibria}, that is, Wardrop equilibria that allow only part of the exogenous flow to enter and traverse the network.

\subsection{Contribution}
We claim three contributions in this paper.
First, we propose a novel type of selfish routing games on parallel networks, which is more suitable for modeling road networks based on the CTM. In this model, links are treated as cells with capacity constraints that depend on the density within the cell, and link travel times are increasing functions of density rather than flow. Second, we characterize the Wardrop equilibria (WE) and the social optima (SO) of this game and prove their essential uniqueness. Third, we introduce the concept of partially transferring WE and demonstrate that under certain conditions, the unique WE of the game can be partially transferring, even when the exogenous demand on the network is less than the min-cut capacity. This provides further evidence of the inefficiency of selfish routing.

\subsection{Related work}\label{sec:related}

The literature already includes some recent works addressing capacity constraints and congested traffic regimes, such as those proposed in \cite{ttimes,krichene2014,pedarsani}. In \cite{ttimes}, the authors propose a static traffic assignment model using density-based travel time functions. In \cite{krichene2014}, the authors analyze a Stackelberg routing game on a parallel network, where a central authority can control a fraction of the total traffic demand to improve the total cost on the network, thus improving efficiency. In \cite{pedarsani}, a mixed-autonomy traffic model is developed, proposing an optimal strategy to provide financial incentives for autonomous vehicles to steer traffic toward efficient equilibria. However, these models differ from ours as they do not involve any supply and demand mechanism. 

The maximum throughput that can be successfully transferred through the network equals its min-cut capacity \cite{shannon,Ford_Fulkerson_1956}.
The problem of identifying routing policies that prompt fully transferring flow allocations has received significant attention in the last years, mostly through dynamical models \cite{como2013a,como2013b,savla2014,como2015,dahleh2018}. In particular, in \cite{como2015} the authors study the behavior of a dynamical network flow model governed by distributed local routing policies allowed to depend on the network state. These policies are characterized by routing decisions at each non-destination node being made independently based only on the state of incoming and outgoing links, without considering the state of other nodes in the network. Nonetheless, capacity constraints are only applied at the exits of the links, allowing any amount of flow to enter a link. It is shown that if the exogenous flow which the network is subject to does not exceed the min-cut capacity, then the class of monotone distributed routing policies ensures that the system globally asymptotically converges to a state where the flow is fully transferred.

In \cite{coogan2015,lovisari}, the authors propose a dynamical network flow model encompassing the CTM with fixed routing policies (not necessarily fully transferring) at non-destination nodes. The main results concern convergence to equilibria. In \cite{coogan2015}, the authors develop a ramp metering control strategy for maximizing the transferred flow.

\subsection{Paper organization} Section~\ref{ch:ctm:sec:modeling} delves into the details of the proposed network structure. Here, we describe the mechanism for supply and demand on each link, define what constitutes valid traffic assignments, and introduce factors influencing travel time based on link density. Section~\ref{ch:ctm:sec:game} establishes a non-atomic selfish routing game on the network. We comprehensively analyze the Wardrop equilibria and social optima of this game, identifying necessary and sufficient conditions for the occurrence of partially transferring Wardrop equilibria. In Section~\ref{ch:ctm:beyond}, we provide an example showing that the problem of partial demand transfer occurs also in more complex network topologies. Finally, Section~\ref{ch:ctm:sec:conclusion} concludes the paper with some closing remarks.

\section{Network modeling}\label{ch:ctm:sec:modeling}
We consider a parallel network consisting of a single OD pair and $N$ distinct non-intersecting routes connecting them. Each route $p$ is composed of $n_p$ links. The network is subject to a constant positive exogenous flow of vehicles $\Phi$ that distributes among its routes. In the following, we describe the functioning of each network link in relation to the traffic density within it.   


\subsection{Characterization of the network links}
Given a link $l\in\mathcal{L}$, let $x_l$ (veh/km) and $f_l$ (veh/h) indicate its density, corresponding to the number of vehicles per unit of length, and its flow, corresponding to the number of vehicles per unit of time. Let $\overline{x}_l$ (veh/km), $\overline{f}_l$ (veh/h), $v_l$ (km/h) and $L_l$ (km) be positive finite constants representing the \emph{jam density} (maximum attainable density), the \emph{capacity} (maximum attainable flow), the \emph{free-flow speed} and the \emph{length} of the link. Now, associate with each link supply and demand functions $s_l(x_l),\ d_l(x_l)$, depending on the link density. 
Supply and demand functions are inspired by Daganzo's fundamental diagram \cite{funddiag} and take the following form:
\begin{equation}
    \label{ch:ctm:supp}
     s_l(x_l)= \min\{\overline{f}_l,w_l(\overline{x}_l-x_l)\},
\end{equation}
\begin{equation}
    \label{ch:ctm:dem}
    d_l(x_l)=\min\{v_lx_l,\overline{f}_l\}.
\end{equation}
Both functions are continuous and piece-wise linear. The supply function is non-increasing with density, reflecting the fact that as more vehicles occupy the link, the fewer additional vehicles the link can accommodate. In contrast, the demand function is non-decreasing, meaning that as more vehicles are on the link, the higher the number of vehicles aiming to leave it.
If we define the \emph{critical density} of a link as $x_l^c:=\overline{f}_l/v_l$, then $w_l=\overline{f}_l/(\overline{x}_l-x_l^c)$, so as to guarantee
\begin{center}
    $s_l(x_l^c)=d_l(x_l^c)=\overline{f}_l$.
\end{center}
When $x_{l}\leq x_l^c$, we say that the link is in \emph{free-flow} or \emph{free-flow regime}, whereas if $x_{l}>x_l^c$, the link is said to be \emph{congested} or in \emph{congested regime}.

In the following section, we characterize the feasible density and flow vectors for a network whose links exhibit such a supply and demand mechanism.

\begin{figure*}
    \centering
    \includegraphics[width=.95\textwidth]{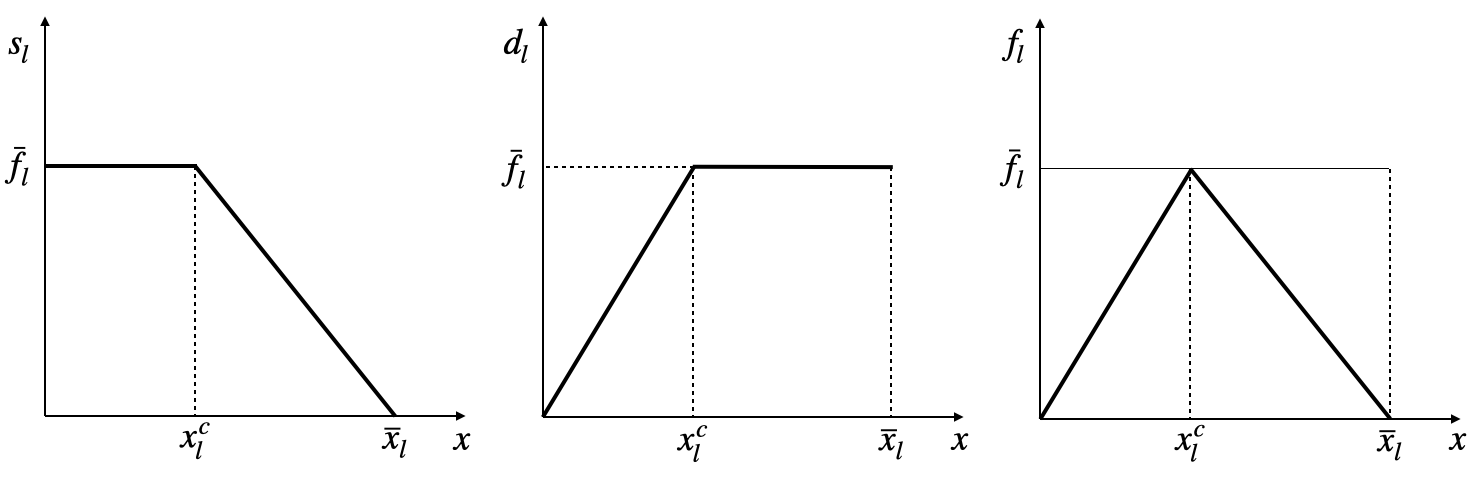}
    \caption{From left to right: supply function as in \eqref{ch:ctm:supp}, demand function as in \eqref{ch:ctm:dem}, Daganzo's fundamental diagram of traffic.}
    \label{fig:fund}
\end{figure*}

\subsection{Traffic assignments}

\begin{defi}
A vector $R=(R_1,\dots,R_N)\in\mathbb{R}_{\geq0}^N$ such that $\sum_{i=1}^NR_i=1$ is called a \emph{routing vector}. Each element of $R$ is called \emph{routing ratio}.
\label{fa}
\end{defi}
The supply and demand functions determine the exchange flow at the interface between contiguous cells. Let $x\in\mathbb{R}_{\geq0}^{n_1+\dots+n_N},\ f\in\mathbb{R}_{\geq0}^{p+n_1+\dots+n_N}$ be the \emph{density} and \emph{flow vectors}, respectively. Consider a route $p$ and two of its consecutive links, $l_p$ and $(l+1)_p$. Then, the inflow from link $l_p$ to link $(l+1)_p$ is
\begin{equation}
    f_{l_p}(x)=\min\{d_{l_p}(x_{l_p}),s_{(l+1)_p}(x_{(l+1)_p})\}.
    \label{inter}
\end{equation}
The inflow of the first link of a route is
\begin{equation}
    f_{0_p}(x)=\min\{\Phi R_p,s_{0_p}(x_{0_p})\},
    \label{start}
\end{equation}
Finally, since the final link of each route is not connected to any other link,
\begin{equation}
    f_{n_p}(x)=d_{n_p}(x_{n_p}).
    \label{end}
\end{equation}
With abuse of notation, we will indicate the density and flow vectors associated with route $p$ as $x_p\in\mathbb{R}_{\geq0}^{n_p},\ f_p\in\mathbb{R}_{\geq0}^{n_p+1}$.

Given a routing vector, we are interested in identifying all the density vectors associated with it.

\begin{figure}
\vspace{.5cm}
    \centering
    \includegraphics[width=.45\textwidth]{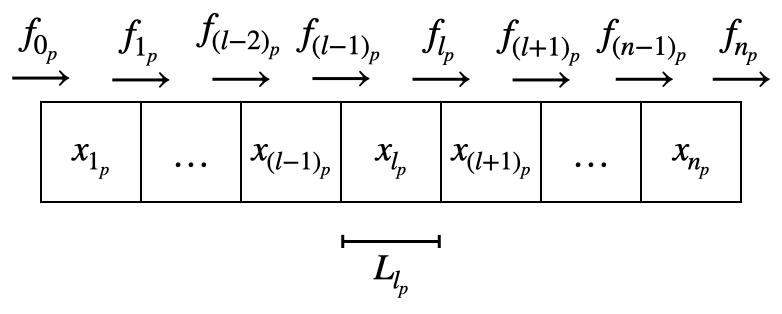}
    \caption{Structure of route $p$.}
    \label{fig:route}
\end{figure}

\begin{defi}
Given a routing vector $R$, a \emph{consistent density vector} $x^R$ is a density vector satisfying to
\begin{equation}
    f_{0_p}(x^R)=f_{1_p}(x^R)= \dots =f_{n_p}(x^R),\quad p=1,\dots,N.
    \label{ch:ctm:idv}
\end{equation}
Let $C(R)$ be the set of all consistent density vectors associated with $R$.
\end{defi}
Consistent density vectors are density vectors such that the inflow and the outflow of each link correspond and this flow is equal for all the links of the route. 
\begin{defi}
    A \emph{traffic assignment} is a pair $(R,x^R)$, where $R$ is routing vector and $x^R$ is a consistent density vector of it.
\end{defi}

Depending on the routing vector, the associated consistent density vectors might be characterized by some congested links or not. \begin{defi}
    The \emph{capacity of route} $p$, $\overline{z}_p$, is the capacity of the route's lowest capacity link:
    \begin{equation}
        \overline{z}_p:=\min_{l\in p}\overline{f}_l.
    \end{equation}
\end{defi}
Given an exogenous flow $\Phi$ and a routing vector $R$, consider the following sets:
    \begin{equation}
    \begin{aligned}
    P_F&=\{p\in\{1,\dots,N\}|\,\Phi R_p<\overline{z}_p\},\\
    P_C&=\{p\in\{1,\dots,N\}|\,\Phi R_p=\overline{z}_p\},\\
    P_S&=\{p\in\{1,\dots,N\}|\,\Phi R_p>\overline{z}_p\}.
    \end{aligned}
    \end{equation}
The set $P_F$ consists of the routes assigned a fraction of exogenous flow smaller than their capacity. The set $P_C$ consists of the routes assigned a fraction of exogenous flow equal to their capacity. Finally, The set $P_S$ consists of the routes assigned a fraction of exogenous flow exceeding their capacity. Then, let us discuss about the shape of the consistent density vectors for a given routing vector.

To ease the discussion, we assume that each route has a unique link of minimum capacity.
\begin{assum}
    Route $p$ has a unique minimum capacity link $b_p\in\{1_p,\dots,n_p\}$, $p=1,\dots,N$.
    \label{ass:ub}
\end{assum}
 This assumption, while introducing a loss of generality, does not undermine the relevance of the findings. Instead, it allows us to focus on specific aspects of the problem and draw conclusions that are still applicable to a wide range of scenarios.

\begin{prop}
    Let Assumption \ref{ass:ub} hold. 
    Then, $\forall p\in P_F$, there exists a unique consistent density vector for route $p$ which is as follows:
    \begin{equation}
    \begin{aligned}
        &x_{l_p}^R=\frac{\Phi R_p}{v_{l_p}},\quad\ \qquad l_p=1_p,\dots,n_p,\\
        &f_{l_p}(x^R)=\Phi R_p,\qquad l_p=1_p,\dots,n_p;
    \end{aligned}
        \label{ffd}
    \end{equation}
    Instead, $\forall p\in P_S$, there exists a unique consistent density vector for route $p$ which is as follows:
    \begin{equation}
        \begin{aligned}
        &x_{l_p}^R=\overline{x}_{l_p}-\frac{\overline{z}_p}{w_{l_p}},\qquad\ \,\qquad l_p=1_p,\dots,(b-1)_p\\
        &x_{l_p}^R=\frac{\overline{z}_p}{v_{l_p}},\qquad \qquad\quad\quad\ \ \ \, l_p=b_p,\dots,n_p,\\
        &f_{l_p}(x^R)=\overline{z}_p,\qquad\qquad\qquad l_p=1_p,\dots,n_p;
        \end{aligned}
        \label{vd}
    \end{equation}
    Finally, $\forall p\in P_C$, a consistent density vector on $p$ is any vector such that, given $k_p\in\{1,\dots,(b-1)_p\}$:
    \begin{equation}
        \begin{aligned}
        &x_{l_p}^R=\frac{\overline{z}_p}{v_{l_p}},\qquad \qquad\quad\ \ \, l_p=1_p,\dots,(k-1)_p;\\
        &x_{k_p}^R\in\biggl[x_{k_p}^c,\overline{x}_{l_q}-\frac{\overline{z}_p}{w_{l_p}}\biggl],\\
        &x_{l_p}^R=\overline{x}_{l_p}-\frac{\overline{z}_p}{w_{l_p}},\,\ \ \,\quad\quad\, l_p=(k+1)_p,\dots,(b-1)_p,\\
        &x_{l_p}^R=\frac{\overline{z}_p}{v_{l_p}},\qquad \qquad\quad\ \ \, l_p=b_p,\dots,n_p,\\
        &f_{l_p}(x^R)=\overline{z}_p,\qquad\qquad\ \, l_p=1_p,\dots,n_p.
        \end{aligned}
        \label{sd}
    \end{equation}
    \label{prop:shapeidv}
\end{prop}

\begin{rem}
    It follows from \eqref{ch:ctm:supp} and \eqref{ch:ctm:dem} that, for any link $l_p$,
    \begin{equation*}
        x_{l_p}^R=\frac{\Phi R_p}{v_{l_p}}<x_{l_p}^c,\qquad x_{l_p}^R=\overline{x}_{l_p}-\frac{\overline{z}_p}{w_{l_p}}>x_{l_p}^c.
    \end{equation*}
    This means that if $p\in P_F$, then all links of route $p$ are in free-flow regime. Contrarily, if $p\in P_S$, then the first $b-1$ links of route $p$ are in congested regime. Finally, if $p\in P_C$, depending on the value of $k_p$, the route might present links with congested regime or not, extending backward from the least capacity link to the origin.
\end{rem}

\begin{proof}
    The proof is split into three parts, each for one of the sets $P_F$, $P_C$ and $P_S$. 
\begin{enumerate}
    \item\label{point1} Consider a route $p\in P_F$. Suppose that $f_{0_p}=\Phi R_p$. Then, from \eqref{ch:ctm:idv}, $f_{l_p}=\Phi R_p,\ l_p=1_p,\dots,n_p$. This implies that  
    \begin{equation*}
        d_{l_p}(x_{l_p})=\Phi R_p\ \text{   or   }\ s_{(l+1)_p}(x_{(l+1)_p})=\Phi R_p.
    \end{equation*}
Suppose that $f_{l_p}=d_{l_p},\ l_p=1_p,\dots,n_p$. Then, it is straightforward that \eqref{ffd} is the only possible density vector with this form satisfying to \eqref{ch:ctm:idv}.
Suppose now that there exists $k_p\in p$ such that $f_{k_p}(x_p)=s_{(k+1)_p}(x_{(k+1)_p})$. Since 
    \begin{equation*}
        f_{k_p}(x_p)=\Phi R_p<\overline{f}_{(k+1)_p},
    \end{equation*}
    it must be that
    \begin{equation*}
        x_{(k+1)_p}>x_{(k+1)_p}^c,
    \end{equation*}
    so that
    \begin{equation*}
        s_{(k+1)_p}(x_{(k+1)_p})<d_{(k+1)_p}(x_{(k+1)_p}).
    \end{equation*}
    From \eqref{ch:ctm:idv}, the last inequality implies that $f_{(k+1)_p}(x_p)=s_{(k+2)_p}(x_{(k+2)_p})$. The same argument can be applied inductively to the subsequent route links, up to the final link of the route. Nevertheless, since $f_{(n-1)_p}(x_p)=s_{n_p}(x_{n_p})$, then $x_{n_p}>x_{n_p}^c$, which in turn implies that outflow of link $n_p$ is equal to $\overline{f}_{n_p}$. As this violates \eqref{ch:ctm:idv}, we proved that there exists no consistent density vector where some links are in congested regime. Hence, the consistent density vector is unique and as in \eqref{ffd}. Using the same argument, it follows that any density vector such that $f_{0_p}<\Phi R_p$ is not a consistent density vector.
    \item \label{point2}Consider a route $p\in P_S$. Clearly, since $\Phi R_p>\overline{z}_p$, only part of $\Phi R_p$ can be accommodated. Suppose that $f_{0_p}=\overline{z}_p$. This imposes that all links preceding $b_p$, which have higher capacity, must be in congested regime so as to guarantee that the flow transferred from a link to the following is $\overline{z}_p$:
    \begin{equation*}
        x_{l_p}=\overline{x}_{l_p}-\frac{\overline{z}_p}{w_{l_p}},\quad f_{(l-1)_p}(x)=s_{l_p}(x_{l_p})=\overline{z}_p,
    \end{equation*}
    $l_p=1_p,\dots,(b-1)_p$. Then, the density on $b_p$ must be equal to $\overline{z}_p/v_{l_p}$. 
    As for the links from $(b+1)_p$ to $n_q$, one can apply the same argument as in \ref{point1}. to the sub-route they form. Again, similarly to \ref{point1}, density vectors such that $f_{0_p}<\overline{z}_p$ are not consistent density vectors, as they do not fulfil to \eqref{ch:ctm:idv}. Finally, density vectors such that $f_{0_p}>\overline{z}_p$ cannot be consistent, as this implies $f_{0_p}>\overline{f}_{b_p}$, which contradicts \eqref{ch:ctm:idv}.
    \item 
Consider a route \(p \in P_C\). Since \(\Phi R_p = \overline{z}_p\), all the inflow can be accommodated. It is easy to verify that all density vectors of the form as in \eqref{sd} satisfy \eqref{ch:ctm:idv} and entirely accommodate \(\Phi R_p\). All consistent density vectors cannot take any different form. For the same argument in \ref{point1} and \ref{point2}, links from \((b+1)_q\) to \(n_q\) cannot be in a congested regime. As for the links preceding \(b_q\), from \eqref{inter}, any congested link must be followed by a congested link that limits the incoming flow from its predecessor to be equal to \(\overline{z}_p\). Finally, also in this case density vectors such that $f_{0_p}<\overline{z}_p$ are not consistent density vectors, as they do not fulfil to \eqref{ch:ctm:idv}. 
\end{enumerate}
\end{proof}

\begin{figure}
    \centering
    \includegraphics[width=.25\textwidth]{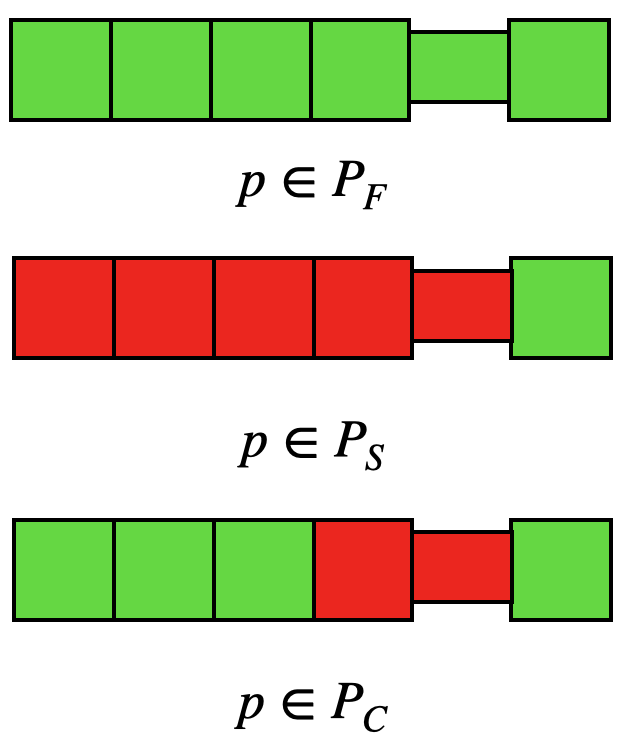}
    \caption{Congestion patterns for routes in $P_F$, $P_C$, $P_S$. Green links are in free-flow regime, red links are congested.}
    \label{fig:coongestion}
\end{figure}


Proposition \ref{prop:shapeidv} prescribes that when a routing vector $R$ violates the capacity constraints of some routes, i.e., $P_S\neq\emptyset$, then the unique consistent density vector associated with it is characterized by congested links. As implied by equation \eqref{sd}, these traffic assignments transfer only a fraction of the exogenous flow directed to that route. We call such traffic assignments \emph{partially transferring}. On the other hand, traffic assignments such that $P_S$ are called \emph{fully transferring}. Given an exogenous flow $\Phi$ exceeding the min-cut capacity of the network, which in our case simply corresponds to the sum of all route capacities, clearly all of its traffic assignments are partially transferring. Therefore, we turn our attention to exogenous flows that do not exceed the min-cut capacity of $\mathcal{G}$.
\begin{assum}[]
    The exogenous flow $\Phi$ is no greater than the min-cut capacity of $\mathcal{G}$:
    \begin{equation*}
        \Phi\leq\sum_{p=1}^N\overline{z}_p.
    \end{equation*}
    \label{mincut}
\end{assum}
Although any $\Phi$ satisfying to Assumption \ref{mincut} admits a fully transferring traffic assignment, such exogenous flows still admit partially transferring traffic assignments, in general.

\begin{figure}
    \centering
    \includegraphics[width=.4\textwidth]{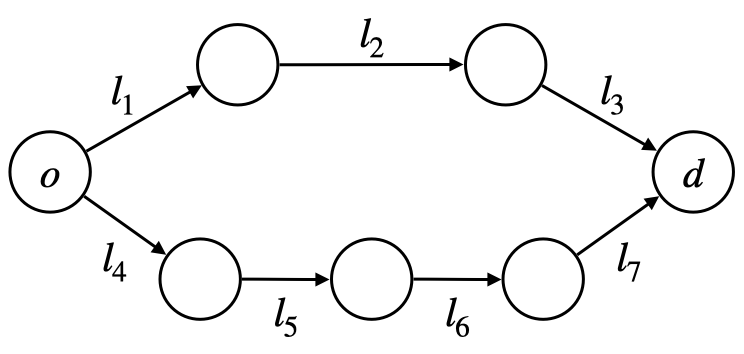}
    \caption{A two-route network. Route $1$ consists of three links, whereas Route $2$ of four.}
    \label{fig:example}
\end{figure}

\begin{ex}[Partially transferring traffic assignment]
    Consider the network in Figure \ref{fig:example}, and assume it is characterized as follows:
    \begin{equation}
        \begin{aligned}
            &\overline{f}=(1500,1500,1000,1500,1500,1500,1500),\\
            &\overline{z}=(1000,1500),\\
            &\overline{x}=(187.5,187.5,100,187.5,187.5,187.5,187.5),\\
            &v_l=40,\ \forall l\in\mathcal{L}.
        \end{aligned}
        \label{param}
    \end{equation}
    Suppose that the network is subject to a constant exogenous flow $\Phi=1500$, which satisfies Assumption~\ref{mincut}. 
    Consider the three following routing vectors:
    \begin{equation*}
        R^{(1)}=(1/3,2/3),\quad R^{(2)}=(3/4,1/4)\quad R^{(3)}=(2/3,1/3).
    \end{equation*}
    When $\Phi$ is assigned according to $R^{(1)}$, then both routes belong to $P_F$, and the unique consistent density vector is
    \begin{equation*}
        x^{R^{(1)}}=(12.5,12.5,12.5,25,25,25,25).
    \end{equation*}
    Hence, the traffic assignment $(R^{(1)},x^{R^{(1)}})$ is unique and fully transferring.
    
    When $\Phi$ is assigned according to $R^{(2)}$, then route $1$ belongs to $P_S$, whereas route $2$ to $P_F$. The unique consistent density vector associated with this routing vector is
    \begin{equation*}
        x^{R^{(2)}}=(87.5,87.5,25,9.375,9.375,9.375,9.375).
    \end{equation*}
    The assignment $(R^{(2)},x^{R^{(2)}})$ is clearly partially transferring, and the amount of flow that does not get transferred equals $375$ veh/h. 

    Finally, when $\Phi$ is assigned according to $R^{(3)}$, then route $1$ belongs to $P_C$, whereas route $2$ to $P_F$. In this case, there exist multiple consistent density vectors, which take one of the two following form:
    \begin{equation*}
    \begin{aligned}
        &x^{R^{(3)}}=(25,x_{l_2}\in[25,87.5],25,12.5,12.5,12.5,12.5),\\
        &x^{R^{(3)}}=(x_{l_1}\in[25,87.5],87.5,25,12.5,12.5,12.5,12.5).
        \end{aligned}
    \end{equation*}
    In this case, all possible traffic assignments $(R^{(3)},x^{R^{(3)}})$ are fully transferring. \qed
\end{ex}

\subsection{Link travel times} 

\begin{figure}
    \centering
    \includegraphics[width=.45\textwidth]{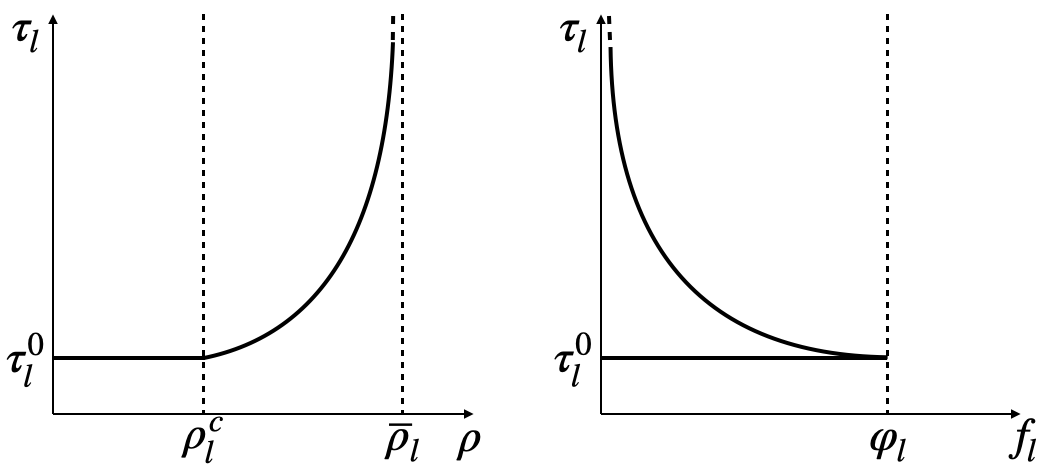}
    \caption{Link travel time as a function of the link density (left) and the relationship between link travel time and link flow (right).}
    \label{fig:time}
\end{figure}
Consistently with the fundamental diagram of traffic, we model travel times as in \cite{krichene2014}:
\begin{equation}
\begin{aligned}
    \tau_{l}:[0,\overline{x}_{l}]&\rightarrow\mathbb{R}_{>0}\cup+\infty\\
    x_l&\mapsto \tau_l(x_{l})=L_{l}\frac{x_{l}}{f_{l}(x_{l})}
\end{aligned},\quad l\in\mathcal{L},
\end{equation}
where $L_{l}$ is the length of link $l$. From the shape of the travel time functions, for a link $l_p$, when $f_{l_p}(x_{l_p})=d_{l_p}(x_{l_p})$, then $\tau_{l_p}(x_{l_p})=L_{l_p}/v_{l_p}$. Thus, when the link is in free-flow, its travel time is constant and equal to the free-flow travel time $\tau_{l_p}^F:=L_{l_p}/v_{l_p}$. On the contrary, when $f_{l_p}(x_{l_p})=s_{l_p}(x_{l_p})$, $\tau_{l_p}\geq\tau_{l_p}^F$ and it is increasing in $x_{l_p}$.  

The travel time of each route $\tau_p(x)$ is simply defined as the sum of the link travel times of all route links:
\begin{equation}
    \tau_p(x_p)=\sum_{l\in p}\tau_l(x_l),\quad p=1,\dots,N.
\end{equation}
For a given routing vector, the travel time of a route $p$, depending on which among the sets $P_F$, $P_S$ and $P_C$ it belongs to, will be as follows:
\begin{itemize}
    \item if $p\in P_F$, then the route attains the lowest possible value of travel time, the \emph{free-flow route travel time}:
\begin{equation*}
    \tau_p^F:=\sum_{l\in p}\tau_l^F; 
\end{equation*}
\item if $p\in P_S$, then the route attains the following value of travel time:
\begin{equation*}
    \tau_p^S:=\sum_{l=1_p}^{(b-1)_p}\tau_l\left(\overline{x}_{l}-\frac{\overline{z}_p}{w_{l}}\right)+\sum_{l=b_p}^{n_p}\tau_l^F; 
\end{equation*}
\item if $p\in P_C$, then the route can attain any value of travel time between $\tau_p^F$ and $\tau_p^S$, precisely
\begin{equation*}
\begin{aligned}
    \tau_p(x_p)=&\sum_{l=1_p}^{(k-1)_p}\tau_l^F+\tau_{k_q}(x_{k_q}^R)+\\&+\sum_{l=(k+1)_p}^{(b-1)_p}\tau_l\left(\overline{x}_{l}-\frac{\overline{z}_p}{w_{l}}\right)+\sum_{l=b_p}^{n_p}\tau_l^F.
\end{aligned}
\end{equation*}
\end{itemize}
Before moving to the next section, it proves convenient to define the following quantities. Given a route $p\in\{1,\dots,N\}$, with abuse of notation, let $\tau_p^{-1}:[\tau_p^F,\tau_p^S]\rightarrow\mathbb{R}_{\geq0}^{n_p}$ be the function that, given $\tau\in[\tau_p^F,\tau_p^S]$, returns a unique consistent density vector $x_p^\tau$ of the form \eqref{sd} such that 
\begin{equation*}
    \tau_p(x_p^\tau)=\tau.
\end{equation*}

\section{Non-atomic routing game (NRG)}\label{ch:ctm:sec:game}
Let us indicate the NRG as $(\mathcal{G},\Phi)$.  Each vehicle chooses its route to minimize its travel time according to the link travel time functions $\tau_l$. 
We assume that the routes are ordered by increasing free-flow travel time, and, to ease the discussion, the travel times $\tau_p^S$ are assumed to be distinct.
\begin{assum}\label{ass3}
    The free-flow travel times $\tau_p^F$ and the travel times $\tau_p^S$ are all distinct, and routes are ordered by increasing free-flow travel times:
    \begin{equation}
    \begin{aligned}
        & \tau_1^F<\tau_2^F<\dots<\tau_N^F,\\
        & \tau_p^S\neq \tau_q^S, \ \forall p,q\in\{1,\dots,N\}.
    \end{aligned} 
    \end{equation}
\end{assum}


\subsection{Wardrop equilibria}
Now, we formalize the notion of Wardrop equilibrium in our setting. 

\begin{defi}[Wardrop equilibrium]
A \emph{Wardrop equilibrium (WE)} of the NRG $(\mathcal{G},\Phi)$ is a traffic assignment $(R^W,x^W)$ such that
\begin{equation}
    R_p^W>0\ \Rightarrow\ \tau_p(x^W)\leq\tau_q(x^W),\quad \forall q=1,\dots,N.
    \label{wc}
\end{equation}
\end{defi}
The following result states characterizes the WE of the NRG $(\mathcal{G},\Phi)$, and establishes whether they are fully or partially transferring. 
In order to state it, let us define
    \begin{equation*}
        k:=\min\biggl\{p\in\{1,\dots,N\}\,\biggl|\,\Phi-\sum_{j=1}^{k}\overline{z}_j\leq0\biggl\},
        \end{equation*}
        \begin{equation*}
        U:=\bigl\{p\in\{1,\dots,k\}\,\bigl|\,\tau_p^S\leq \tau_j^F,\text{ for some }j\leq k\bigl\}.
        \end{equation*}
The index $k$ represents how many of the most efficient routes, i.e., routes with smallest travel time, it takes to fully accommodate the exogenous flow $\Phi$, while the set $U$ consists of those routes such that their free-flow travel time exceeds $\tau_p^S$, for some $p$ among the first $k$ most efficient routes. 

\begin{thm}[Characterization of the WEs]
    Consider the NRG $(\mathcal{G},\Phi)$ and suppose that Assumption \ref{ass:ub}, \ref{mincut} and \ref{ass3} hold. Then, if $U=\emptyset$, the game admits a unique fully transferring WE $(R^W,x^W)$, such that
        \begin{equation}
           \label{ftw}
           \begin{aligned}
               &\Phi R_p^W=\begin{cases}
                    \overline{z}_p, & p=1,\dots,k-1\\
                    \Phi-\sum_{p=1}^{z-1}\overline{z}_p, & p=k\\
                    0, & p=k+1,\dots,N
                \end{cases},\\&\\
               &x_p^W=\tau_{p}^{-1}(\tau_k^F),\quad p=1,\dots k-1,\\
               &x_k^W=\left(\frac{\Phi-\sum_{p=1}^{k-1}\overline{z}_p}{v_{1_k}},\dots,\frac{\Phi-\sum_{p=1}^{k-1}\overline{z}_p}{v_{n_k}}\right),
           \end{aligned}
        \end{equation}
        and all used routes share the same travel time $\tau_k^F$.
        
       If $U\neq\emptyset$, let $u:=\min U$ and let 
       \begin{equation*}
           j:=\min\{p=u+1,\dots,k\,|\,\tau_p^F\geq \tau_u^S\}.
       \end{equation*}
       Then, \begin{itemize}
       \item if $\tau_j^F>\tau_u^S$, then the game admits a unique partially transferring WE $(R^W,x^W)$, such that
           \begin{equation}
           \begin{aligned}
           &\Phi R_p^W=\begin{cases}
                    \overline{z}_p, & p\in\{1,\dots,j-1\}\setminus u\\
                    \Phi-\sum_{p=1,p\neq u}^{j-1}\overline{z}_p, & p=u\\
                    0, & p=j,\dots,N
                \end{cases},
           \\
           &x_p^W=\tau_p^{-1}(\tau_u^S),\quad p=1,\dots,j-1,
           \end{aligned}
           \label{ptf2}
       \end{equation}
       and all used routes share the same travel time $\tau_u^S$;
       \item if $\tau_j^F=\tau_u^S$, then $(R^W,x^W)$ reads
           \begin{equation}
           \begin{aligned}
           &\Phi R_p^W=\begin{cases}
                    \overline{z}_p, & p\in\{1,\dots,j-1\}\setminus u\\
                   \Phi R_u^W, & p=u\\
                   \Phi(1-\sum_{i=1}^{j-1}R_j^W), &p=j\\
                    0, & p=j+1,\dots,N
                    
                \end{cases},\\&  \Phi R_u^W\in\left[\Phi-\sum_{p=1,p\neq u}^{j}\overline{z}_p,\Phi-\sum_{p=1,p\neq u}^{j-1}\overline{z}_p\right], \\
           &x_p^W=\tau_p^{-1}(\tau_u^S),\quad p=1,\dots,j-1,\\
           &x_j^W=\left(\frac{\Phi(1-\sum_{i=1}^{j-1}R_j^W)}{v_{1_j}},\dots,\frac{\Phi(1-\sum_{i=1}^{j-1}R_j^W)}{v_{n_j}}\right).
           \end{aligned}
           \label{ptf1}
       \end{equation}
       If $j<k$, then $(R^W,x^W)$ is partially transferring, and all used routes hare the same travel time $\tau_j^F=\tau_u^S$. If $j=k$, then $(R^W,x^W)$ is fully transferring if and only if $\Phi R_u^W=\overline{z}_u$, and all used routes hare the same travel time $\tau_k^F=\tau_u^S$. 
       \end{itemize}
    \label{thm:we}
\end{thm}
    Before proving the theorem, we provide the reader with some intermediate results.
\begin{lem}
    Suppose $(R^W,x^W)$ is a WE of $(\mathcal{G},\Phi)$. If $R^W_p>0$, then $\Phi R^W_q\geq\overline{z}_q,\ \forall q<p$.
    \label{lem:aux1}
\end{lem}
\begin{proof}
    By contradiction, assume that $\Phi R^W_q<\overline{z}_q$ for some $q<p$. Then, for any consistent density vector $x^W$ of $R^W$,  $\tau_q(x^W)=\tau_q^F<\tau_p^F\leq\tau_q(x^W)$, which contradicts \eqref{wc}.
\end{proof}
\begin{lem}
    Suppose $(R^W,x^W)$ is a WE of $(\mathcal{G},\Phi)$. Then, $\mathrm{supp}(R^W)\subseteq\{1,\dots,k\}$.
    \label{lem:aux2}
\end{lem}
\begin{proof}
    By contradiction, suppose that $\max\mathrm{supp}(R^W)>k$. By Lemma \ref{lem:aux1}, it should be that $\Phi R^W_q\geq\overline{z}_q,\ \forall q<\max\mathrm{supp}(R^W)$, which contradicts the definition of $k$.
\end{proof}

We are now ready to provide the proof of Theorem \ref{thm:we}.
\begin{proof}
    Lemmata \ref{lem:aux1} and \ref{lem:aux2} imply that any WE of the game has support of the form $\{1,\dots,p\},\ p\leq k$. We split the proof into three parts: the first part is dedicated to characterize the WE of $(\mathcal{G},\Phi)$ when $U=\emptyset$, the second one addresses the case $U\neq\emptyset$ and $\tau_j^F>\tau_u^S$, and the third one the case $U\neq\emptyset$ and $\tau_j^F=\tau_u^S$. 
    \begin{enumerate}
        \item $U=\emptyset$: in this case, there cannot be any routes such that $\Phi R_q^W>\overline{z}_q$, as this would imply that 
        \begin{equation*}
            \tau_q(x_q^W)=\tau_q^S>\tau_k^F,\quad \forall q<k,
        \end{equation*}
        contradicting the Wardrop condition \eqref{wc}. This also implies that 
        $\mathrm{supp}(R^W)=\{1\dots,k\}$, as if $\mathrm{supp}(R^W)=\{1\dots,p\}$, with $p<k$, them by definition of $k$, there should exist $q\in\{1\dots,p\}$ such that $\Phi R_q^W>\overline{z}_q$. 
        By combining these facts with Lemma~\ref{lem:aux1}, it becomes straightforward that the only possible traffic assignment $(R^W,x^W)$ satisfying to the Wardrop condition \eqref{wc} is that in \eqref{ftw}. Clearly, the traffic assignment in \eqref{ftw} is fully transferring.
        
        \item $U\neq\emptyset$, $\tau_j^F>\tau_u^S$: we start by observing that $\mathrm{supp}(R^W)\subseteq\{1\dots,j-1\}$, as all routes $q\in\{j,\dots,k\}$ have free-flow travel time greater than $\tau_u^S$. Observe also that there cannot be any routes such that $\Phi R_q^W>\overline{z}_q,\ q\in\{1\dots,j-1\}\setminus u$. In fact, as $u=\min U$ and all maximum route travel time are distinct, it holds that
        \begin{equation*}
            \tau_u^S=\min_p\tau_p^S.
        \end{equation*}
        Thus, $\Phi R_q^W>\overline{z}_q,\ q\in\{1\dots,j-1\}\setminus u$, would imply 
        \begin{equation*}
            \tau_q(x_q^W)=\tau_q^S>\tau_u^S,
        \end{equation*}
        violating the Wardrop condition \eqref{wc}.
        These facts, combined with Lemma~\ref{lem:aux1}, imply that the only possible traffic assignment $(R^W,x^W)$ satisfying to the Wardrop condition \eqref{wc} is that in \eqref{ptf2}.
        
       \item $U\neq\emptyset$, $\tau_z^F=\tau_u^S$: analogously to the previous case,  $\mathrm{supp}(R^W)\subseteq\{1\dots,j\}$, as all routes $q\in\{j+1,\dots,k\}$ have free-flow travel time greater than $\tau_u^S$, and there cannot be any routes such that $\Phi R_q^W>\overline{z}_q,\ q\in\{1\dots,j\}\setminus u$, as it would result in contradicting the Wardrop condition \eqref{wc}. By combining these facts with Lemma~\ref{lem:aux1}, it follows that all traffic assignments $(R^W,x^W)$ that take the form in \eqref{ptf1} satisfy to the Wardrop condition \eqref{wc}. As all such routing vectors satisfy to $\Phi R_u^W\geq\overline{z}_u$ and $\Phi R_j^W\leq\overline{z}_j$, they attain the maximum travel time on route $u$ and the free-flow travel time on route $j$. Among these traffic assignments, it is straightforward to see that the only one which is fully transferring is the one associated with the case $j=k$ and such that $\Phi R_u^W=\overline{z}_u$.
    \end{enumerate}
\end{proof}

Theorem~\ref{thm:we} highlights a potential drawback of selfish routing: \emph{partially transferring Wardrop equilibria}. Even when the network is subject to an exogenous flow smaller than its min-cut capacity, users' selfish behavior can lead to traffic assignments that only partially transfer the exogenous demand. In a sense, we might think of this as selfish routing reducing the effective capacity of the network, as vehicles would never use routes that are sub-optimal in terms of travel time. Because all users aim for the shortest travel time routes and share the same queue before entering the network, the exogenous flow may be accommodated only partly, leading to congestion at the origin.
In the following, we characterize the exact amount of exogenous flow loss due to partial demand transfer.

\begin{corol}
    Consider a partially transferring \emph{WE} $(R^W,x^W)$. Let us indicate $\Psi$ the amount of non-transferred exogenous flow. Then:
    \begin{itemize}
        \item if $(R^W,x^W)$ takes the form in \eqref{ptf2}, then 
        \begin{equation}
            \Psi=\Phi-\sum_{p=1}^{j-1}\overline{z}_p;
        \end{equation}
        \item if $(R^W,x^W)$ takes the form in \eqref{ptf1}, then:
        \begin{itemize}
            \item if $j<k$, then 
            \begin{equation}
                \Psi\in\left[\Phi-\sum_{p=1}^{j}\overline{z}_p,\Phi-\sum_{p=1}^{j-1}\overline{z}_p\right];
            \end{equation}
            \item if $j=k$, then 
            \begin{equation}
                \Psi\in\left[0,\Phi-\sum_{p=1}^{k-1}\overline{z}_p\right];
            \end{equation}
        \end{itemize}
    \end{itemize}
    \label{cor31}
\end{corol}

\begin{ex}\label{ch:ctm:ex2}
    Consider the network in Figure \ref{fig:example} with capacities, jam densities and speeds as in \eqref{param}. Suppose also that link lengths are as follows:
    \begin{equation*}
        L=\left(1,1,0.5,2,2,2,2\right).
        \label{ch:ctm:eql}
    \end{equation*}
    Suppose that $\Phi=1000$, so that Assumption \ref{ass:ub} is satisfied. In this case, $k=1$ and $U=\emptyset$, so the unique \emph{WE} $(R^W,x^W)$ is
    \begin{equation*}
        R^W=\left(1,0\right),\quad x^W=\left(25,25,25,0,0,0,0\right).
    \end{equation*}
    Now, assume that $\Phi=1500$, which still satisfies to Assumption \ref{ass:ub}. In this other case, $k=2$, but $U=\{1\}$, since
    \begin{equation*}
        \tau_1^S=11.25\text{ min }<12\text{ min }=\tau_2^F.
    \end{equation*}
    As a result, the unique \emph{WE} of the game is the following partially transferring traffic assignment:
    \begin{equation*}
        R^W=(1,0),\quad x^W=\left(87.5,87.5,25,0,0,0,0\right).
    \end{equation*}
    The amount of non-transferred flow $\Psi$ amounts to $500$. 
\end{ex}

    Wardrop equilibria are said to be \emph{essentially unique} when they all share the same minimum travel time \cite{rough}. Theorem \ref{thm:we} implies that the game $(\mathcal{G},\Phi)$ exhibits essential uniqueness. Specifically, when \( U = \emptyset \), the WE is unique. When \( U \neq \emptyset \), if \( \tau_j^F > \tau_u^S \), the WE is unique; however, if \( \tau_j^F = \tau_u^S \), the WE is not unique, but all WEs have the same travel time.

\begin{rem}
    Assumptions~\ref{ass:ub} and \ref{ass3} were made to simplify the analysis of the Wardrop Equilibria (WEs) of $(\mathcal{G},\Phi)$. Assumption~\ref{ass:ub} certainly limits the generality of the model. Without Assumption~\ref{ass:ub}, routes can be characterized by multiple minimal capacity links. In this more general case, routes would have multiple bottlenecks, and the categories of valid density vectors for routes in sets $P_C$ and $P_S$ would become richer, encompassing a wider variety of congestion patterns. On the other hand, Assumption~\ref{ass3} imposes minimal limitations on the set of parameters. We underscore that these two assumptions allow for capturing the problem of partial demand transfer and are not the cause of it. As we will show in one of the next sections with an example, this issue also presents in networks where these two assumptions do not hold.
    
\end{rem}

\begin{rem}[Comparison with \cite{krichene2014}]
As mentioned in Section~\ref{sec:related}, a non-atomic selfish routing game relying on a description of the traffic state based both on density and flow, accounting for capacity constraints and congested traffic regimes has already been proposed in \cite{krichene2014}, but that model does not include a supply and demand mechanism. This leads to two important differences. First, our model exhibits essential uniqueness, whereas that model does not. Second, in some cases, that model does not admit a WE for certain values of exogenous flow, even when the latter is less than the min-cut capacity of the network. In contrast, our model admits a WE for any possible exogenous flow.
\end{rem}

\subsection{Social optimum}
In general, a \emph{social optimum} is an assignment minimizing some system cost. Here, we provide a definition of social optimum that accounts for both the minimization of the total travel time over the network and the full transfer of the exogenous flow~$\Phi$.
\begin{defi}
    Given an exogenous flow $\Phi$ satisfying to Assumptions \ref{ass:ub} and \ref{mincut}, a \emph{social optimum} (SO) of the game $(\mathcal{G},\Phi)$ is a traffic assignment $(R^O,x^O)$  such that
\begin{equation}
    \begin{aligned}
(R^O,x^O)=\argmin_{z,x} & \sum_{p=1}^N\Phi R_p\tau_p(x)\\
\textrm{s.t.} \quad & x\in C(R),\\
  &\Phi R_p\leq\overline{z}_p, \\
  &\sum_{p=1}^NR_p=1.
  \end{aligned}
  \label{so}
\end{equation}
    \end{defi}
\noindent We prove  that, in our setting, there exists a unique SO. 
\begin{prop}
    Suppose that Assumption \ref{ass:ub} is satisfied. Then, the NRG $(\mathcal{G},\Phi)$ admits a unique SO $(R^O,x^O)$, whose expression is as follows:
    \begin{equation}
           \begin{aligned}
               &\Phi R^O=\left(\overline{z}_1,\dots,\overline{z}_{(k-1)},\Phi-\sum_{p=1}^{k-1}\overline{z}_p,0\dots,0\right),\\
               &x_p^O=\left(\frac{\overline{z}_p}{v_{1_p}},\dots,\frac{\overline{z}_p}{v_{n_p}}\right),\quad p=1,\dots,k-1,\\
               &x_k^O=\left(\frac{\Phi-\sum_{p=1}^{k-1}\overline{z}_p}{v_{1_k}},\dots,\frac{\Phi-\sum_{p=1}^{k-1}\overline{z}_p}{v_{n_k}}\right).
           \end{aligned}
           \label{soex}
        \end{equation}
\end{prop}
\begin{proof}
    Suppose that (${R^O}'$,${x_p^O}'$) is a social optimum of the NRG and suppose that ${R^O}'>0$. It is straightforward that every link of route $p$ is in free-flow. In fact, since $\Phi{R^O}'\leq\overline{z}_p$, $p\in F\cup S$, which means ${R^O}'$ admits a consistent flow vector such that all links are in free-flow. Hence, ${x_p^O}'$ cannot present saturated links, as otherwise would not be minimizing the cost in \eqref{so}. Then, \eqref{so} reduces to
    \begin{equation*}
    \begin{aligned}
(R^O,x^O)=\argmin_{z,x} & \sum_{p=1}^N\Phi R_p\tau_p^F\\
\textrm{s.t.} \quad & x_p=\left(\frac{\Phi R_{p}}{v_{l_p}}\right)_{l_p=1_p}^{n_p},\ p=1,\dots,N,\\
  &\Phi R_p\leq\overline{z}_p. \\
  &\sum_{p=1}^NR_p=1.
  \end{aligned}
\end{equation*}
It follows immediately that the unique SO is the one using the first $k$ routes as in \eqref{soex}.
\end{proof}
One of the measure most commonly used to quantify the inefficiency of WEs in routing games is the Price of Anarchy (PoA) \cite{rough}. The PoA of a WE corresponds to the total travel time realized by the WE and the minimum total travel time achievable, the one realized by the SO: 
\begin{equation}
    \poa(R^W,x^W)=\frac{\sum_{p=1}^N\Phi R_p^W\tau_p(x^W)}{\sum_{p=1}^N\Phi R_p^O\tau_p(x^O)}.
\end{equation}
In our model the PoA turns out not to be the most appropriate measure of inefficiency. In fact, for partially transferring WEs, the PoA loses its significance, as the WE is transferring a flow less than that transferred by the SO. In this case, a WE might even realize a total travel time smaller than the SO, but this comes from the fact that the WE is transferring less flow.
On the other hand, when $(R^W,x^W)$ is fully transferring, the PoA is well-defined and takes the following form:
\begin{equation*}
    \poa(R^W,x^W)=\frac{\Phi\cdot\tau_k^F}{\sum_{p=1}^{k-1}\overline{z}_p\cdot\tau_p^F+\left(\Phi-\sum_{p=1}^{k-1}\overline{z}_p\right)\tau_k^F}\geq1.
\end{equation*}

Another interesting fact to remark is that if the WE of the NRG is fully transferring, then the WE and the SO share the same routing vector, i.e., $R^W=R^O$ (see \eqref{ftw} and \eqref{soex}). As this might sound contradictory, let us discuss it more in detail.
\begin{ex}\label{ch:ctm:ex3}
    Consider the network in Figure \ref{fig:example} with capacities, jam densities and speeds as in \eqref{param} and link lengths
    \begin{equation*}
        L=\left(1.5,1.5,1.5,2,2,2,2\right).
    \end{equation*}
    as in \eqref{ch:ctm:eql}. Assume that $\Phi=1500$, so that Assumption \ref{ass:ub} is satisfied. The \emph{WE} in this case is unique and corresponds to  
    \begin{equation*}
        R^W=\left(\frac{2}{3},\frac{1}{3}\right),\ \ x^W=\left(25,83.3,25,12.5,12.5,12.5,12.5\right).
        \label{rxwex}
    \end{equation*}
    Indeed, such traffic assignment implies that the two route travel times satisfy two
    \begin{equation}
        \tau_1(x^W)=\tau_2(x^W)= 12\text{ min}.
    \end{equation}
    On the other hand, the \emph{SO} corresponds to 
    \begin{equation*}
        R^O=\left(\frac{2}{3},\frac{1}{3}\right),\quad x^O=\left(25,25,25,12.5,12.5,12.5,12.5\right).
    \end{equation*}
    In this case, 
    \begin{equation*}
        \poa(R^W,x^W)=\frac{24}{17}
    \end{equation*}
    The \emph{SO} fully transfer the whole exogenous demand, while also minimizing the total travel time, keeping all used routes in free-flow regime. We can provide the following explanation to this phenomenon. At WE, each user selfishly chooses their route to minimize their own travel time. This selfish behavior leads to a density vector $x^W$ as given in \eqref{rxwex}. Consequently, the flow at the origin is split between the two roads in a way that results in the routing vector $R^W$. Conversely, at SO, the objective is to minimize the overall travel time for all users. A central planner determines the optimal routing vector $R^O$, which results in a specific density vector $x^O$. The density vector $x^O$ ensures that all traffic routes used are in the free-flow regime, meaning they are not congested.
\end{ex}

Therefore, even though the two routing vectors coincide, \( R^W \) can be seen as the routing vector induced by the Wardrop condition to ensure that the used routes have the same travel time, while \( R^O \) is the routing vector that induces an optimal utilization of the network.

\section{Beyond parallel networks}\label{ch:ctm:beyond}
In the previous sections, we analyzed selfish routing on parallel networks. This section aims to provide an example showing that selfish routing can cause the same type of issues, such as partial demand transfer, in more complex network topologies beyond parallel networks. 
\begin{figure}
    \centering
    \includegraphics[width=0.8\linewidth]{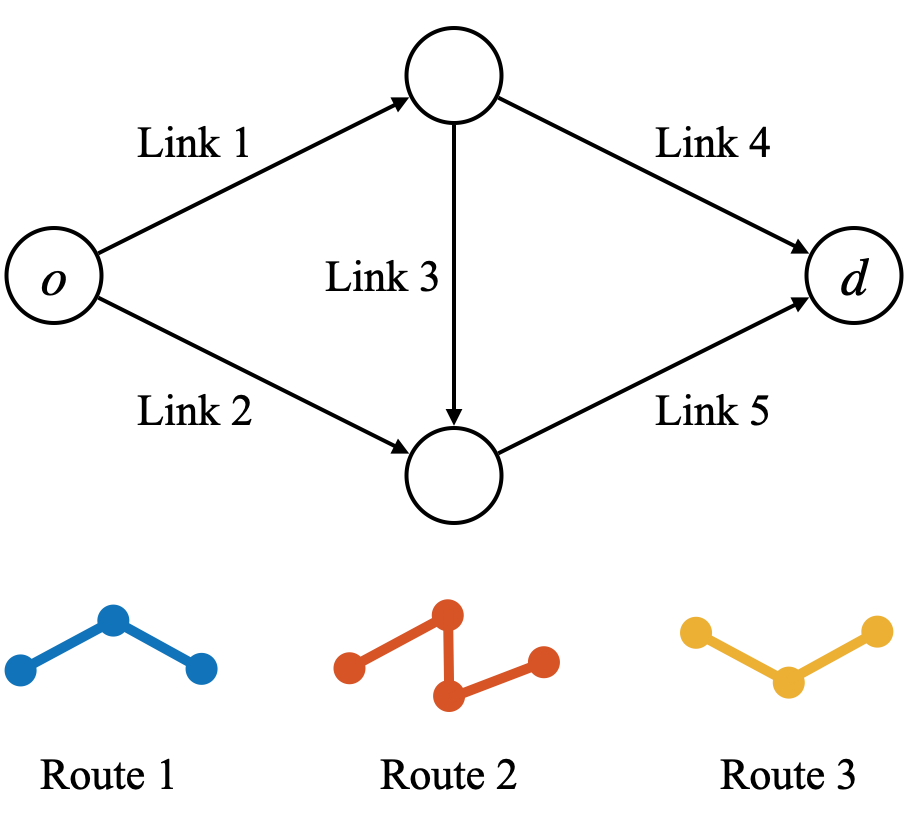}
    \caption{Wheatstone's network.}
    \label{wheat2}
\end{figure}
Consider the network in Figure~\ref{wheat2} and suppose that the network geometry is the following:
\begin{equation*}
\begin{aligned}
    &\overline{f}=\left(1500,1500,800,1500,1500\right),\\
    &v_l=40,\quad \forall l \in \mathcal{L},\\
    &\overline{x}=(187.5,187.5,100,187.5,187.5),\\
    &L=(8,16,4,16,8).
\end{aligned}   
\end{equation*}

Suppose that the network is subject to an exogenous flow $\Phi=1600$. A WE for this network is given by
\begin{equation*}
    R^W=\left(0,\frac{9}{16},\frac{7}{16}\right),\qquad x^W=\left(107.5,51.41,20,0,37.5\right),
\end{equation*}
where the routing vector $R^W$ represents the fraction of exogenous flow allocated on each route.
From the expression of $x^W$, one can see that both link $1$ and $2$ are in congested regime. The travel times of the used routes, Route $2$ and Route $3$, is $1$ h $23$ min, approximately. The travel time of the unused route, Route $1$, is $1$ h $29$ min, instead. Then, one can notice that $(R^W,x^W)$ is a partially transferring. Indeed, from \eqref{ch:ctm:supp}, the supply of link one is exactly $800$ veh/h. As the fraction of exogenous flow aiming to use Route $2$ must pass through link $1$, it is clear that the exogenous flow cannot be fully accommodated. Also in this case, users' selfish behavior leads to an inefficient traffic pattern that causes partial demand transfer.

This example demonstrates that partial demand transfer is a fundamental issue of selfish routing. Moreover, its occurrence is not limited to parallel networks but can also arise in more complex network topologies.


\section{Concluding remarks}\label{ch:ctm:sec:conclusion}

The main contribution of this paper lies in the analysis of the selfish routing model in a network subject to supply and demand constraints on its links, inspired by Daganzo's cell transmission model. This approach effectively characterizes the congestion phenomena typical of traffic networks. Our analysis highlights that the issues associated with selfish routing extend beyond a mere reduction in traffic efficiency in terms of total travel time. We have indeed demonstrated that selfish routing can lead to sub-optimal utilization of the road network's capacity. Even when the network is subject to an exogenous flow less than its min-cut capacity, which can theoretically be fully transferred across the network, the traffic distribution caused by the selfish behavior of users results in only part of the traffic being transferred, leaving part of the exogenous flow unserved at the network's origin.

This study opens several avenues for further research. The first potential extension involves applying the model to more complex network topologies beyond the family of parallel networks. This would significantly enhance the model's applicability to real-world scenarios. The main challenge in generalizing to arbitrary networks lies in computing the Wardrop equilibria. In the current setting, we found that these equilibria can be written in close form, and an algorithm for their computation is straightforward. Certainly, this is not the case for more complex networks. Therefore, a primary future objective will be to determine if the calculation of Wardrop equilibria can be framed as an optimization problem similar to how Wardrop equilibria are calculated in the classical routing games formulation \cite{patriksson}.

A second important extension involves analyzing scenarios where the management of exogenous flow at the network's origin differs from what we have considered. Our assumption of a single origin for the exogenous flow implies that users feeds into a single queue before entering the network. This implies that players aiming for different routes will accumulate at this common queue, independent of the route they aim for. While this situation may correspond to certain real-world scenarios, there are also cases that fall outside this framework and would be better modeled if each route had its own queue, i.e., the entry to one road does not depend on the others. In such cases, the travel time for each route should account for the waiting time in the queue to enter that route. In this case, the problem of partial demand transfer would probably be mitigated, as excessively long waiting times for one route would prompt users to consider alternative routes.

Future work should also aim to encompass heterogeneity,  so as to account for users with different levels of information or preferences, and mixed behaviors, to capture the presence of user classes that act coordinately. 

A further extension of the model involves its dynamization. As the current model is entirely static, making it unclear whether traffic dynamics actually converge to these traffic assignments. To address this, we need to design dynamic network flows based on CTM principles, similar to the approaches used in previous works by \cite{coogan2015,lovisari}. However, it is crucial to incorporate routing policies that reflect the selfish behavior of users.

\bibliographystyle{IEEEtran} 
\bibliography{biblio} 

\begin{thebibliography}{10}
\providecommand{\url}[1]{#1}
\csname url@samestyle\endcsname
\providecommand{\newblock}{\relax}
\providecommand{\bibinfo}[2]{#2}
\providecommand{\BIBentrySTDinterwordspacing}{\spaceskip=0pt\relax}
\providecommand{\BIBentryALTinterwordstretchfactor}{4}
\providecommand{\BIBentryALTinterwordspacing}{\spaceskip=\fontdimen2\font plus
\BIBentryALTinterwordstretchfactor\fontdimen3\font minus \fontdimen4\font\relax}
\providecommand{\BIBforeignlanguage}[2]{{%
\expandafter\ifx\csname l@#1\endcsname\relax
\typeout{** WARNING: IEEEtran.bst: No hyphenation pattern has been}%
\typeout{** loaded for the language `#1'. Using the pattern for}%
\typeout{** the default language instead.}%
\else
\language=\csname l@#1\endcsname
\fi
#2}}
\providecommand{\BIBdecl}{\relax}
\BIBdecl

\bibitem{wardrop52}
J.~G. Wardrop, ``Some theoretical aspects of road traffic research,'' \emph{Proceedings of the Institute of Civil Engineers, Part {II}}, vol.~1, pp. 325--378, 1952.

\bibitem{rough}
N.~Nisan, T.~Roughgarden, E.~Tardos, and V.~V. Vazirani, \emph{Algorithmic Game Theory}.\hskip 1em plus 0.5em minus 0.4em\relax New York, NY, USA: Cambridge University Press, 2007.

\bibitem{th:cabannes}
T.~C.~P. Cabannes, ``The impact of information-aware routing on road traffic. from case studies to game-theoretical analysis and simulations,'' Ph.D. dissertation, University of California, Berkeley, 2022.

\bibitem{ibp}
D.~Acemoglu, A.~Makhdoumi, A.~Malekian, and A.~Ozdaglar, ``Informational {B}raess’ paradox: The effect of information on traffic congestion,'' \emph{Operations Research}, vol.~66, no.~4, pp. 893--917, 2018.

\bibitem{thai}
J.~Thai, N.~Laurent-Brouty, and A.~M. Bayen, ``Negative externalities of {G}{P}{S}-enabled routing applications: A game theoretical approach,'' in \emph{2016 IEEE 19th International Conference on Intelligent Transportation Systems (ITSC)}, 2016, pp. 595--601.

\bibitem{funddiag}
M.~Carey and M.~Bowers, ``A review of properties of flow–density functions,'' \emph{Transport Reviews}, vol.~32, no.~1, pp. 49--73, 2012.

\bibitem{ttimes}
P.~Kachroo and S.~Sastry, ``Traffic assignment using a density-based travel-time function for intelligent transportation systems,'' \emph{IEEE Transactions on Intelligent Transportation Systems}, vol.~17, no.~5, pp. 1438--1447, 2016.

\bibitem{daganzo1994}
C.~F. Daganzo, ``The cell transmission model: A dynamic representation of highway traffic consistent with the hydrodynamic theory,'' \emph{Transportation Research Part B: Methodological}, vol.~28, pp. 269--287, 1994.

\bibitem{daganzo1995}
------, ``The cell transmission model, part {I}{I}: Network traffic,'' \emph{Transportation Research Part B: Methodological}, vol.~29, no.~2, pp. 79--93, 1995.

\bibitem{krichene2014}
W.~Krichene, J.~D. Reilly, S.~Amin, and A.~M. Bayen, ``Stackelberg routing on parallel networks with horizontal queues,'' \emph{IEEE Transactions on Automatic Control}, vol.~59, no.~3, pp. 714--727, 2014.

\bibitem{pedarsani}
E.~Bıyık, D.~A. Lazar, R.~Pedarsani, and D.~Sadigh, ``Incentivizing efficient equilibria in traffic networks with mixed autonomy,'' \emph{IEEE Transactions on Control of Network Systems}, vol.~8, no.~4, pp. 1717--1729, 2021.

\bibitem{shannon}
P.~Elias, A.~Feinstein, and C.~Shannon, ``A note on the maximum flow through a network,'' \emph{IRE Transactions on Information Theory}, vol.~2, no.~4, pp. 117--119, 1956.

\bibitem{Ford_Fulkerson_1956}
L.~R. Ford and D.~R. Fulkerson, ``Maximal flow through a network,'' \emph{Canadian Journal of Mathematics}, vol.~8, p. 399–404, 1956.

\bibitem{como2013a}
G.~Como, K.~Savla, D.~Acemoglu, M.~A. Dahleh, and E.~Frazzoli, ``Robust distributed routing in dynamical networks — {P}art {I}: Locally responsive policies and weak resilience,'' \emph{IEEE Transactions on Automatic Control}, vol.~58, no.~2, pp. 317--332, 2013.

\bibitem{como2013b}
------, ``Robust distributed routing in dynamical networks – {P}art {I}{I}: Strong resilience, equilibrium selection and cascaded failures,'' \emph{IEEE Transactions on Automatic Control}, vol.~58, no.~2, pp. 333--348, 2013.

\bibitem{savla2014}
K.~Savla, G.~Como, and M.~A. Dahleh, ``Robust network routing under cascading failures,'' \emph{IEEE Transactions on Network Science and Engineering}, vol.~1, no.~1, pp. 53--66, 2014.

\bibitem{como2015}
G.~Como, E.~Lovisari, and K.~Savla, ``Throughput optimality and overload behavior of dynamical flow networks under monotone distributed routing,'' \emph{IEEE Transactions on Control of Network Systems}, vol.~2, no.~1, pp. 57--67, 2015.

\bibitem{dahleh2018}
A.~Y. Yazıcıoğlu, M.~Roozbehani, and M.~A. Dahleh, ``Resilient control of transportation networks by using variable speed limits,'' \emph{IEEE Transactions on Control of Network Systems}, vol.~5, no.~4, pp. 2011--2022, 2018.

\bibitem{coogan2015}
S.~Coogan and M.~Arcak, ``A compartmental model for traffic networks and its dynamical behavior,'' \emph{IEEE Transactions on Automatic Control}, vol.~60, no.~10, pp. 2698--2703, 2015.

\bibitem{lovisari}
E.~Lovisari, G.~Como, and K.~Savla, ``Stability of monotone dynamical flow networks,'' in \emph{2014 53rd IEEE Conference on Decision and Control (CDC)}, 2014, pp. 2384--2389.

\bibitem{patriksson}
M.~Patriksson, \emph{The traffic assignment problem: models and methods}.\hskip 1em plus 0.5em minus 0.4em\relax Courier Dover Publications, 2015.

\end{thebibliography}

\vskip 0pt plus -1fil

\begin{IEEEbiography}[{\includegraphics[width=1in,height =1.25in,clip,keepaspectratio]{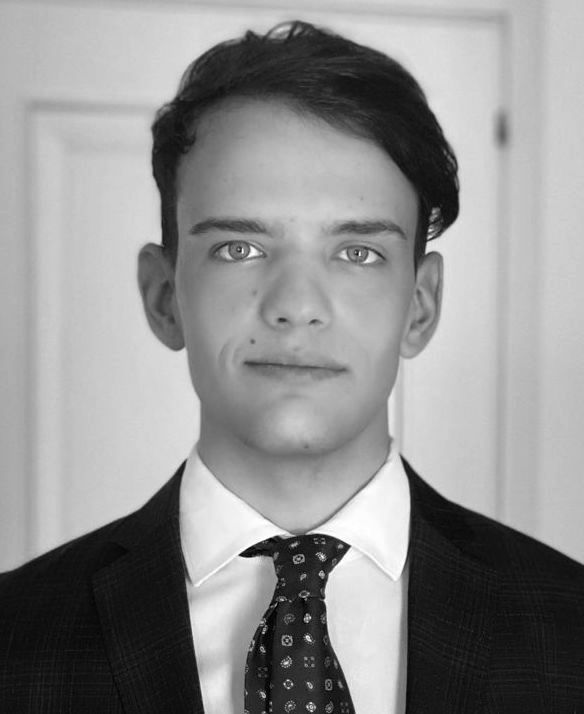}}]{Tommaso Toso} Tommaso Toso (Student Member, IEEE) received the B.S. degree in Mathematics for Engineering, M.S. degree in Mathematical Engineering (cum Laude), in 2019 and 2021, respectively, from the Politecnico di Torino, Turin, Italy. He is curretly a Ph.D. student at the GIPSA-Lab, Grenoble, France. His research focuses on dynamics and control in network systems, with applications to transportation networks.
\end{IEEEbiography}
\vskip 0pt plus -1fil

\begin{IEEEbiography}[{\includegraphics[width=1in,height =1.25in,clip,keepaspectratio]{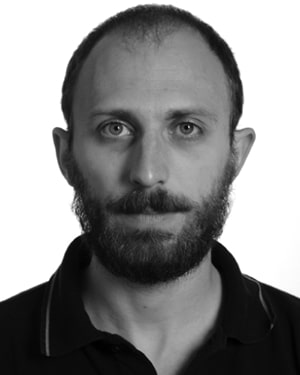}}]{Paolo Frasca} Paolo Frasca (Senior Member, IEEE) received the Ph.D. degree from the Politecnico di Torino, Turin, Italy, in 2009. 
He was an Assistant Professor with the University of Twente, Enschede, Netherlands, from 2013 to 2016. Since October 2016, he has been a CNRS Researcher at GIPSA-Lab, Grenoble, France. His research interests cover the theory of control systems and networks, with main applications in infrastructural and social networks. On these topics, he has (co)authored more than 50 journal publications. He has been Associate Editor for several conferences and journals, including the International Journal of Robust and Nonlinear Control, the IEEE Control Systems Letters, the Asian Journal of Control, and Automatica.
\end{IEEEbiography}
\vskip 0pt plus -1fil

\begin{IEEEbiography}[{\includegraphics[width=1in,height =1.25in,clip,keepaspectratio]{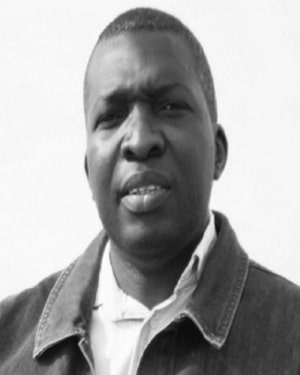}}]{Alain Y. Kibangou} Alain Y. Kibangou (Member, IEEE) received the B.S. degree in physics-oriented electronics, M.S. degree in electrical engineering from Cadi Ayyad University, Marrakesh, Morocco, in 1998 and 2000, respectively, and the Ph.D. degree in control, signal, and image processing jointly from the University of Nice Sophia Antipolis, Nice, France, and the Cadi Ayyad University, Marrakesh, Morocco, in 2005. He was a Postdoctoral Researcher with I3S Laboratory; with LAAS, Toulouse, France; and with GIPSA-Lab, Grenoble, France. He has been an Associate Professor with the Université Grenoble Alpes, Grenoble, France, and a Researcher with GIPSA-Lab since 2009. His research interests include network systems analysis, distributed estimation, traffic prediction and estimation, nonlinear filtering, identification, and tensor analysis.
\end{IEEEbiography}

\end{document}